\documentclass{article}

\usepackage{amsfonts}
\usepackage{amsmath}
\usepackage{amsthm}

\newcommand{\R}{\mathbb{R}}

\DeclareMathOperator{\dom}{dom}

\newtheorem{proposition}{Proposition}

\newtheorem*{theorem*}{Theorem}
\theoremstyle{definition}

\theoremstyle{remark}

\title{A primal representation of the Monge-Kantorovich norm}

\author{
  D\'avid Terj\'ek \\
  Alfr\'ed R\'enyi Institute of Mathematics\\
  Budapest, Hungary\\
  \texttt{dterjek@renyi.hu} \\
}

\begin{document}

\maketitle

\begin{abstract}
In this note, following \cite{Chitescuetal2014}, we show that the Monge-Kantorovich norm on the vector space of countably additive measures on a compact metric space has a primal representation analogous to the Hanin norm, meaning that similarly to the Hanin norm, the Monge-Kantorovich norm can be seen as an extension of the Kantorovich-Rubinstein norm from the vector subspace of zero-charge measures, implying a number of novel results, such as the equivalence of the Monge-Kantorovich and Hanin norms.
\end{abstract}

As presented in \cite{Chitescuetal2014} and summarized in \cite{Cobzasetal2019}, the vector space $cabv(X)$ of countably additive measures with bounded variation on a compact metric space $(X,d)$ can be normed in at least two distinct ways to induce the topology of weak convergence of measures (also called weak-$*$ convergence of measures) on subsets bounded with respect to the total variation norm. One $(cabv(X),\Vert . \Vert_{MK})$ is by the Monge-Kantorovich norm (also called the Kantorovich-Rubinstein norm)
\begin{equation}
\Vert \mu \Vert_{MK} = \sup_{f \in Lip(X), \Vert f \Vert_{sum} \leq 1}\left\{\int f d\mu \right\},
\end{equation}
with $(Lip(X),\Vert . \Vert_{sum})$ being the vector space of functions $f : X \to \R$ for which the Lipschitz seminorm
\begin{equation}
\Vert f \Vert_L = \sup_{x, y \in X, x \neq y}{\left\{\frac{\vert f(x) - f(y) \vert}{d(x, y)}\right\}}
\end{equation}
is finite, normed by the sum norm, which is defined as
\begin{equation}
\Vert f \Vert_{sum} = \Vert f \Vert_L + \Vert f \Vert_\infty.
\end{equation}
The other $(cabv(X),\Vert . \Vert_H)$ is by the Hanin norm
\begin{equation}
\Vert \mu \Vert_H = \inf_{\nu \in cabv(X,0)}\left\{ \Vert \nu \Vert_{KR} + \Vert \mu - \nu \Vert_{TV} \right\},
\end{equation}
where $\Vert . \Vert_{TV}$ is the total variation norm, and $(cabv(X,0),\Vert . \Vert_{KR})$ is the vector subspace $cabv(X,0)=\{\mu\in cabv(X) : \mu(X)=0\}$ of zero-charge measures normed by the Kantorovich-Rubinstein norm (also called the modified Kantorovich-Rubinstein norm)
\begin{equation}
\Vert \mu \Vert_{KR} = \sup_{f \in Lip(X,x_0), \Vert f \Vert_L \leq 1}\left\{\int f d\mu \right\},
\end{equation}
with $(Lip(X,x_0),\Vert . \Vert_L)$ being the vector subspace $Lip(X,x_0) = \{f\in Lip(X) : f(x_0)=0\}$ of Lipschitz continuous functions vanishing at an arbitrary fixed point $x_0 \in X$, normed by the Lipschitz seminorm.

The topological dual of $(cabv(X,0),\Vert . \Vert_{KR})$ and $(Lip(X,x_0),\Vert . \Vert_L)$ are isometrically isomorphic. The topological dual of $(cabv(X),\Vert . \Vert_H)$ and $(Lip(X),\Vert . \Vert_{max})$ are isometrically isomorphic as well, with the max norm defined as
\begin{equation}
\Vert f \Vert_{max} = \max\{ \Vert f \Vert_L , \Vert f \Vert_\infty \},
\end{equation}
leading to the dual representation
\begin{equation}
\Vert \mu \Vert_H = \sup_{f \in Lip(X), \Vert f \Vert_{max} \leq 1}\left\{\int f d\mu \right\}.
\end{equation}

Inspired by the similarity of the dual representation of the Hanin norm and the definition of the Monge-Kantorovich norm, we prove the following theorem, showing that analogously to the Hanin norm, the Monge-Kantorovich norm can be seen as an extension of the Kantorovich-Rubinstein norm from the subspace $cabv(X,0)$ to the whole space $cabv(X)$, leading to a number of consequences.

\begin{theorem*} \label{theorem}
The Monge-Kantorovich norm has the primal representation
\begin{equation}
\Vert \mu \Vert_{MK} = \inf_{\nu \in cabv(X,0)}\left\{ \max\left\{ \Vert \nu \Vert_{KR} , \Vert \mu - \nu \Vert_{TV} \right\} \right\},
\end{equation}
the supremum in the dual representation is achieved as $\exists f \in Lip(X) : \Vert f \Vert_{sum} = 1 \land \Vert \mu \Vert_{MK} = \int f d\mu$, the subset of measures in $cabv(X)$ with finite support is dense in $(cabv(X),\Vert . \Vert_{MK})$, the topological dual of $(cabv(X),\Vert . \Vert_{MK})$ and $(Lip(X), \Vert . \Vert_{sum})$ are isometrically isomorphic, and the norms $\Vert.\Vert_{MK}$ and $\Vert.\Vert_H$ are equivalent as $\Vert \mu \Vert_{MK} \leq \Vert \mu \Vert_H \leq 2 \Vert \mu \Vert_{MK}$ for $\forall \mu \in cabv(X)$.
\end{theorem*}

To obtain the proof, we will apply techniques of convex analysis cited from \cite{Zalinescu2002}. We start with a number of propositions.

\begin{proposition}
Given $\lambda \in [0,1]$, the mapping $F_\lambda : (cabv(X,0)\Vert.\Vert_{KR}) \to \R$ defined as
\begin{equation}
F_\lambda(\nu) = \lambda \Vert \nu \Vert_{KR}
\end{equation}
is proper, convex and continuous, and its convex conjugate $F_\lambda^* :  (Lip(X,x_0),\Vert.\Vert_L) \to \overline{\R}$ is the indicator
\begin{equation}
F_\lambda^*(f) = i_{\{f \in Lip(X,x_0) : \Vert f \Vert_L \leq \lambda\}}(f).
\end{equation}
\end{proposition}
\begin{proof}
By \cite[Corollary~2.4.16]{Zalinescu2002},
\begin{equation}
(\nu \to \Vert \nu \Vert_{KR})^* = (f \to i_{\{f \in Lip(X,x_0) : \Vert f \Vert_L \leq 1\}}(f)).
\end{equation}
By \cite[Theorem~2.3.1(v)]{Zalinescu2002},
\begin{equation}
(\nu \to \lambda \Vert \nu \Vert_{KR})^* = (f \to \lambda i_{\{f \in Lip(X,x_0) : \Vert f \Vert_L \leq 1\}}(\lambda^{-1}f)),
\end{equation}
which is equivalent to the proposed conjugate relation. The mapping is clearly proper, convex and continuous by being the constant multiple of a norm with a positive multiplier.
\end{proof}

\begin{proposition}
Given $\lambda \in [0,1]$ and $\mu \in cabv(X)$, the mapping $G_{\lambda,\mu} : (cabv(X),\Vert.\Vert_H) \to \R$ defined as
\begin{equation}
G_{\lambda,\mu}(\nu) = (1-\lambda) \Vert \mu - \nu \Vert_{TV}
\end{equation}
is proper, convex and lower semicontinuous, and its convex conjugate $G_{\lambda,\mu}^* :  (Lip(X),\Vert.\Vert_{max}) \to \overline{\R}$ is
\begin{equation}
G_\lambda^*(f) = i_{\{f \in Lip(X) : \Vert f \Vert_\infty \leq 1-\lambda\}}(f) - \int f d\mu.
\end{equation}
\end{proposition}
\begin{proof}
By \cite[Theorem~8.4.10]{Cobzasetal2019}, the level sets of the mapping $(\nu \to \Vert \nu \Vert_{TV})$ are compact with respect to topology of the weak convergence of measures, hence compact with respect to the topology induced by $\Vert.\Vert_H$ as well by \cite[Theorem~8.5.7]{Cobzasetal2019}. This implies that the level sets are closed in $(cabv(X),\Vert.\Vert_H)$, hence the mapping is lower semicontinuous, and clearly proper and convex as well. It is also sublinear, so that by \cite[Theorem~2.4.14(i)]{Zalinescu2002} one has the conjugate relation
\begin{equation}
(\nu \to \Vert \nu \Vert_{TV})^* =
i_{\partial (\nu \to \Vert \nu \Vert_{TV})(0)},
\end{equation}
where by definition the subdifferential at $0$ is
\begin{equation}
\partial (\nu \to \Vert \nu \Vert_{TV})(0) =
\{f \in Lip(X) \ \vert \ \forall \mu \in cabv(X) : \int f d\mu \leq \Vert \mu \Vert_{TV} \}.
\end{equation}
Since $X$ is compact, any $f \in Lip(X)$ achieves its minimum and maximum, hence
\begin{equation}
\exists x_0 \in X : \vert f(x_0) \vert = \Vert f \Vert_\infty,
\end{equation}
implying that 
\begin{equation}
\max\left\{ \int f d\mu : \mu \in cabv(X), \Vert \mu \Vert_{TV} = \xi \right\} = \int f d(\pm \xi \delta_{x_0}) = \xi \Vert f \Vert_\infty
\end{equation}
with $\delta_{x_0}$ being the Dirac measure at $x_0$, and the sign of $\xi$ is opposite to that of $f(x_0)$. It follows that
\begin{equation}
\exists \mu \in cabv(X), \Vert \mu \Vert_{TV} = \xi : \int f d\mu > \Vert \mu \Vert_{TV} \iff 
\Vert f \Vert_\infty > 1,
\end{equation}
which is true for any $\xi \geq 0$, implying that
\begin{equation}
\partial (\nu \to \Vert \nu \Vert_{TV})(0) = \{f \in Lip(X) : \Vert f \Vert_\infty \leq 1 \},
\end{equation}
leading to the conjugate relation
\begin{equation}
(\nu \to \Vert \nu \Vert_{TV})^* = i_{\{f \in Lip(X) : \Vert f \Vert_\infty \leq 1\}}.
\end{equation}
By \cite[Theorem~2.3.1(v)]{Zalinescu2002},
\begin{equation}
(\nu \to \Vert -\nu \Vert_{TV})^* = (f \to i_{\{f \in Lip(X) : \Vert f \Vert_\infty \leq 1\}}(-f)),
\end{equation}
where $(-f)$ can clearly be replaced by $(f)$.
By \cite[Theorem~2.3.1(vi)]{Zalinescu2002}, 
\begin{equation}
(\nu \to \Vert \mu-\nu \Vert_{TV})^* = \left(f \to i_{\{f \in Lip(X) : \Vert f \Vert_\infty \leq 1\}}(f) - \int f d\mu\right).
\end{equation}
By \cite[Theorem~2.3.1(v)]{Zalinescu2002}, 
\begin{multline}
(\nu \to (1-\lambda)\Vert \mu-\nu \Vert_{TV})^* = \left(f \to (1-\lambda)i_{\{f \in Lip(X) : \Vert f \Vert_\infty \leq 1\}}((1-\lambda)^{-1}f) 
\right.\vphantom{\int}\\\left.
- (1-\lambda)\int (1-\lambda)^{-1}f d\mu\right),
\end{multline}
which is clearly equivalent to the proposition.
\end{proof}

\begin{proposition}
Given $\mu \in cabv(X)$, the mapping $H_\mu : (cabv(X,0),\Vert.\Vert_{KR}) \times (cabv(X),\Vert.\Vert_H) \to \R$ defined as
\begin{equation}
H_\mu(\nu_1,\nu_2) = \max\{\Vert \nu_1 \Vert_{KR}, \Vert \mu - \nu_2 \Vert_{TV}\}
\end{equation}
is proper, convex and lower semicontinuous, and its convex conjugate $H_\mu^* :  (Lip(X,x_0),\Vert.\Vert_L) \times (Lip(X),\Vert.\Vert_{max}) \to \overline{\R}$ is
\begin{multline}
H_\mu^*(f_1,f_2) = \min_{\lambda \in [0,1]}\left\{ i_{\{f \in Lip(X,x_0) : \Vert f \Vert_L \leq \lambda\}}(f_1) 
\vphantom{\int}\right.\\\left.
+ i_{\{f \in Lip(X) : \Vert f \Vert_\infty \leq 1-\lambda\}}(f_2) - \int f_2 d\mu \right\}.
\end{multline}
\end{proposition}
\begin{proof}
By \cite[Corollary~2.8.12]{Zalinescu2002}, the conjugate relation
\begin{multline}
((\nu_1,\nu_2) \to \max\{\lambda^{-1}F_\lambda(\nu_1), (1-\lambda)^{-1}G_{\lambda,\mu}(\nu_2)\})^* \\= ((f_1,f_2) \to \min_{\lambda \in [0,1]}\{ F_\lambda^*(f_1) + G_{\lambda,\mu}^*(f_2) \})
\end{multline}
holds, which together with the previous propositions gives the claimed conjugate relation. By \cite[Theorem~2.1.3(vii)]{Zalinescu2002}, $H_\mu$ is proper and convex. Clearly the mappings $(((\nu_1,\nu_2)) \to \lambda^{-1}F_\lambda(\nu_1))$ and $(((\nu_1,\nu_2)) \to (1-\lambda)^{-1}G_{\lambda,\mu}(\nu_2))$ are lower semicontinuous, hence $H_\mu$ as well by being their pointwise maximum.
\end{proof}

\begin{proposition}
The mapping $(\nu \to \nu) : (cabv(X,0),\Vert.\Vert_{KR} \to (cabv(X),\Vert.\Vert_H)$ is linear and continuous, and its adjoint $(\nu \to \nu)^*$ is $(f \to f - f(x_0)) : (Lip(X),\Vert.\Vert_{max}) : (Lip(X,x_0),\Vert.\Vert_L)$.
\end{proposition}
\begin{proof}
It is clear that $\Vert \nu \Vert_H \leq \Vert \nu \Vert_{KR}$, so the linear operator $(\nu \to \nu)$ is bounded, hence continuous. For $\forall \nu \in cabv(X,0), f \in Lip(X)$ it holds that $\int (f-f(x_0)) d\nu = \int f d\nu - f(x_0) \nu(X) = \int f d\nu$, proving the adjoint relation.
\end{proof}

We are now ready to prove the theorem.

\begin{proof}[Proof of Theorem~\ref{theorem}]
First we show that the condition \cite[Theorem~2.8.1(iii)]{Zalinescu2002} holds for the mapping $H_\mu$. Since $H_\mu$ is finite everywhere, one has $\dom H_\mu = cabv(X,0) \times cabv(X)$. For any $\nu_1 \in cabv(X,0)$, the restriction of the mapping $H_\mu(\nu_1,\cdot)$ to any closed subset is lower semicontinuous. For any $r>0$, the closed and nonempty subset $\{ \nu \in cabv(X) : \Vert \nu \Vert_{TV} \leq r \}$ is complete with respect to any metric metrizing the topology of the weak convergence of measures by \cite[Theorem~8.4.10]{Cobzasetal2019}, such as the metric induced by $\Vert.\Vert_H$ by \cite[Theorem~8.5.7]{Cobzasetal2019}. Hence the restriction of the mapping $H_\mu(\nu_1,\cdot)$ to this subset is lower semicontinuous, and therefore has points of continuity by \cite[Theorem~1.1]{Sietal2020}, so that $\exists \nu_2 \in cabv(X)$ such that $H_\mu(\nu_1,\cdot)$ is continuous at $\nu_2$.

Since the condition \cite[Theorem~2.8.1(iii)]{Zalinescu2002} is satisfied, by \cite[Corollary~2.8.2]{Zalinescu2002}, one has
\begin{equation}
\inf_{\nu \in cabv(X,0)}\{ H_\mu((\nu,\nu)) \} = \max_{f \in Lip(X)}\{ -H_\mu^*((-(f-f(x_0)),f)) \},
\end{equation}
or equivalently
\begin{multline}
\inf_{\nu \in cabv(X,0)}\{ \max\{\Vert \nu \Vert_{KR}, \Vert \mu - \nu \Vert_{TV}\} \} \\= \max_{f \in Lip(X)}\left\{ -\min_{\lambda \in [0,1]}\left\{ i_{\{f \in Lip(X,x_0) : \Vert f \Vert_L \leq \lambda\}}(-(f-f(x_0))) 
\right.\right.\\\left.\left.\vphantom{\int}
+ i_{\{f \in Lip(X) : \Vert f \Vert_\infty \leq 1-\lambda\}}(f) - \int f d\mu \right\} \right\},
\end{multline}
where, since $\Vert -(f-f(x_0)) \Vert_L = \Vert f \Vert_L$ and $\min_{x}\{g(x)\}=-\max_{x}\{-g(x)\}$, the right side further simplifies to
\begin{multline}
\max_{f \in Lip(X)}\left\{ \max_{\lambda \in [0,1]}\left\{ \int f d\mu - i_{\{f \in Lip(X,x_0) : \Vert f \Vert_L \leq \lambda\}}(f)
\right.\right.\\\left.\left.\vphantom{\int}
- i_{\{f \in Lip(X) : \Vert f \Vert_\infty \leq 1-\lambda\}}(f) \right\} \right\}.
\end{multline}
Given $f \in Lip(X)$, it is clear that
\begin{multline}
\max_{\lambda \in [0,1]}\left\{ \int f d\mu - i_{\{f \in Lip(X,x_0) : \Vert f \Vert_L \leq \lambda\}}(f)
- i_{\{f \in Lip(X) : \Vert f \Vert_\infty \leq 1-\lambda\}}(f) \right\} \\
= \begin{cases}
\int f d\mu \text{ if $\exists \lambda \in [0,1] : \Vert f \Vert_\infty \leq 1-\lambda\ \land \Vert f \Vert_L \leq \lambda$,}\\
-\infty \text{ otherwise.}
\end{cases}
\end{multline}
All we need to show now is that
\begin{equation}
\exists \lambda \in [0,1] : \Vert f \Vert_\infty \leq 1-\lambda\ \land \Vert f \Vert_L \leq \lambda
\iff \Vert f \Vert_{sum} \leq 1.
\end{equation}
If $\Vert f \Vert_\infty \leq 1-\lambda\ \land \Vert f \Vert_L \leq \lambda$, then $\Vert f \Vert_\infty + \Vert f \Vert_L \leq \lambda + (1-\lambda) = 1$. If $\Vert f \Vert_{sum} \leq 1$, then $\lambda = \Vert f \Vert_L$ suffices. Hence one has
\begin{equation}
\inf_{\nu \in cabv(X,0)}\{ \max\{\Vert \nu \Vert_{KR}, \Vert \mu - \nu \Vert_{TV}\} \} = \max_{f \in Lip(X), \Vert f \Vert_{sum} \leq 1}\left\{\int f d\mu \right\},
\end{equation}
which is exactly the primal formula we set out to prove, together with the variant of the dual formula with a maximum instead of a supremum, implying that $\exists f_* \in Lip(X) : \Vert f_* \Vert_{sum} \leq 1 \land \Vert \mu \Vert_{MK} = \int f_* d\mu$. If $\Vert f_* \Vert_{sum} < 1$ for such an $f_*$, then we get the contradiction $\int \Vert f_* \Vert_{sum}^{-1} f_* d\mu > \Vert \mu \Vert_{MK}$, so that one actually has $\Vert f_* \Vert_{sum} = 1$.

For the density claim, notice that the missing ingredient for the proof of \cite[Proposition~8.5.3]{Cobzasetal2019} to work with $\Vert . \Vert_{MK}$ instead of $\Vert . \Vert_H$ is the formula $\forall \nu \in cabv(X,0), \mu \in cabv(X) : \Vert \mu \Vert_{MK} \leq \Vert \nu \Vert_{KR} + \Vert \mu - \nu \Vert_{TV}$, which in light of the primal representation clearly holds.

For the duality claim, notice again that the missing ingredient for the proof of \cite[Proposition~8.5.5]{Cobzasetal2019} to work with $\Vert . \Vert_{MK}$ instead of $\Vert . \Vert_H$ is exactly the density claim we just proved, hence $(cabv(X),\Vert . \Vert_{MK})^* \cong (Lip(X), \Vert . \Vert_{sum})$ holds as well.

For the equivalence claim, notice that on one hand, one has $\Vert \nu \Vert_{KR} + \Vert \mu - \nu \Vert_{TV} \geq \max\{\Vert \nu \Vert_{KR}, \Vert \mu - \nu \Vert_{TV}\}$ for $\forall \nu \in cabv(X,0)$, hence $\Vert \mu \Vert_H \geq \Vert \mu \Vert_{MK}$ for $\forall \mu \in cabv(X)$. On the other hand, one has $2\max\{\Vert \nu \Vert_{KR}, \Vert \mu - \nu \Vert_{TV}\} \geq \Vert \nu \Vert_{KR} + \Vert \mu - \nu \Vert_{TV}$ for $\forall \nu \in cabv(X,0)$, hence $2\Vert \mu \Vert_{MK} \geq \Vert \mu \Vert_{H}$ for $\forall \mu \in cabv(X)$. Therefore the norms $\Vert . \Vert_{MK}$ and $\Vert . \Vert_H$ are equivalent.
\end{proof}

\bibliographystyle{apalike}
\bibliography{kantrub}

\begin{thebibliography}{}

\bibitem[Chitescu et~al., 2014]{Chitescuetal2014}
Chitescu, I., Ioana, L., Miculescu, R., and Nita, L. (2014).
\newblock Monge–kantorovich norms on spaces of vector measures.
\newblock {\em Results in Mathematics}, 70:349--371.

\bibitem[Cobza{\c{s}} et~al., 2019]{Cobzasetal2019}
Cobza{\c{s}}, {\c{S}}., Miculescu, R., and Nicolae, A. (2019).
\newblock {\em Lipschitz Functions}.
\newblock Lecture Notes in Mathematics. Springer International Publishing.

\bibitem[Si and Zhang, 2020]{Sietal2020}
Si, Z. and Zhang, Z. (2020).
\newblock On the continuous points of semi-continuous functions.
\newblock {\em The Journal of Analysis}.

\bibitem[Zalinescu, 2002]{Zalinescu2002}
Zalinescu, C. (2002).
\newblock {\em Convex Analysis in General Vector Spaces}.
\newblock World Scientific.

\end{thebibliography}

\end{document}